\documentclass[12pt]{amsart}

\usepackage{amsmath,amssymb,enumerate,mathtools,graphicx}
\usepackage[pdftex,pagebackref=false]{hyperref} 
\newtheorem{theorem}{Theorem}[section]

\newtheorem{lemma}[theorem]{Lemma}

\newtheorem{corollary}[theorem]{Corollary}

\theoremstyle{definition}

\newcommand{\bR}{\mathbb R}
\newcommand{\bZ}{\mathbb Z}

\newcommand{\cL}{\mathcal{L}}

\newcommand{\cM}{\mathcal{M}}

\newcommand{\cU}{\mathcal{U}}

\newcommand{\princext}[3]{\widehat{\PG}(#1,#2;#3)}
\DeclareMathOperator{\si}{si}

\DeclareMathOperator{\cl}{cl}
\DeclareMathOperator{\PG}{PG}
\DeclareMathOperator{\GF}{GF}
\DeclareMathOperator{\AG}{AG}

\newcommand{\elem}{\varepsilon}
\newcommand{\del}{\!\setminus\!}
\newcommand{\con}{/}

\begin{document}
\sloppy
\author{Peter Nelson}
\address{Department of Combinatorics and Optimization,
University of Waterloo, Waterloo, Canada}
\thanks{ This research was partially supported by a grant from the
Office of Naval Research [N00014-12-1-0031].}
\title[Dense matroids]{Matroids denser than a projective geometry}
\begin{abstract}
	The {\em growth-rate function} for a minor-closed class $\cM$ of matroids is the
function $h$ where, for each non-negative integer $r$,
$h(r)$ is the maximum number of elements of a simple matroid in $\cM$ 
with rank at most $r$. The Growth-Rate Theorem of Geelen, Kabell, Kung, and Whittle
shows, essentially, that  the growth-rate function is always either linear, quadratic, exponential with some prime power $q$ as the base, or infinite.
Morover, if the growth-rate function is exponential with base $q$, then the class contains all $\GF(q)$-representable matroids, and so $h(r)\ge \frac{q^r-1}{q-1}$ for each $r$. We characterise the classes that satisfy $h(r) = \frac{q^r-1}{q-1}$ 
for all sufficiently large $r$. As a consequence, we determine the eventual value of the growth rate function for most classes defined by excluding lines, free spikes and/or free swirls. 

\end{abstract}
\date{\today}
\maketitle
\section{Introduction}
	The \emph{principal extension} of a flat $F$ in a matroid $M$ by an element $e \notin E(M)$ is the matroid $M'$ such that $M = M' \del e$ and $F$ is the unique minimal flat of $M$ for which $e \in \cl_{M'}(F)$. We write $\princext{n-1}{q}{k}$ for the principal extension of a rank-$k$ flat in $\PG(n-1,q)$. We prove the following:
	
	\begin{theorem}\label{main}
		Let $q$ be a prime power and let $\ell \ge 2$ and $n \ge 2$ be integers. If $M$ is a simple  matroid with $|M| > |\PG(r(M)-1,q)|$ and $r(M)$ is sufficiently large, then $M$ has a minor isomorphic to $U_{2,\ell+2}$, $\princext{n-1}{q}{2},\princext{n-1}{q}{n}$, or $\PG(n-1,q')$ for some $q' > q$. 
	\end{theorem}
	
	This result first appeared in [\ref{nthesis}] and essentially follows from material in [\ref{line}], but our proof is much shorter due to the use of the matroidal density Hales-Jewett theorem [\ref{dhjpaper}]. 
	
	Theorem~\ref{main} has several corollaries related to the growth rate functions of minor-closed classes. For a nonempty minor-closed class of matroids $\cM$, the \emph{growth rate function} $h_{\cM}(n): \bZ_0^+ \to \bZ \cup \{\infty\}$ is the function whose value at each integer $n$ is the maximum number of elements in a simple matroid $M \in \cM$ with $r(M) \le n$. Clearly $h_{\cM}(n) = \infty$ for all $n \ge 2$ if $\cM$ contains all simple rank-$2$ matroids; in all other cases, growth rate functions are quite tightly controlled by a theorem of Geelen, Kabell, Kung and Whittle: 
	
	\begin{theorem}[Growth rate theorem]
		Let $\cM$ be a nonempty minor-closed class of matroids not containing all simple rank-$2$ matroids. There exists $c \in \bR$ such that either:
		\begin{enumerate}
			\item $h_{\cM}(n) \le cn$ for all $n$, 
			\item $\binom{n+1}{2} \le h_{\cM}(n) \le cn^2$ for all $n$, and $\cM$ contains all graphic matroids, or
			\item\label{exp} there is a prime power $q$ so that $\frac{q^n-1}{q-1} \le h_{\cM}(n) \le cq^n$ for all $n$, and $\cM$ contains all $\GF(q)$-representable matroids. 
		\end{enumerate}
	\end{theorem}
		
	Our main result thus applies to the densest matroids in every class of type (\ref{exp}) for which the lower bound $h_{\cM}(n) \ge \frac{q^n-1}{q-1}$ does not eventually hold with equality.  
	
	\subsection*{Minor-closed classes}
	
	We now give a version of our main theorem in terms of minor-closed classes, and state several corollaries. For each prime power $q$, let $\cL(q)$ denote the class of $\GF(q)$-representable matroids. Let $\cL^{\circ}(q)$ denote the closure under minors and isomorphism of $\{\princext{n-1}{q}{n} : n \ge 2\}$. Let $\cL^{\lambda}(q)$ denote the closure under minors and isomorphism of $\{\princext{n-1}{q}{2}: n \ge 2\}$. Our main theorem thus implies the following:
	
\begin{theorem}\label{main2}
	Let $q$ be a prime power. If $\cM$ is a minor-closed class of matroids such that $\frac{q^n-1}{q-1} < h_{\cM}(n) < \infty$ for infinitely many $n$, then $\cM$ contains $\cL^{\circ}(q)$, $\cL^{\lambda}(q)$ or $\cL(q')$ for some $q' > q$. 
\end{theorem}

One can easily determine the growth rate functions of $\cL^{\circ}(q)$ and $\cL^{\lambda}(q)$; we have $h_{\cL^{\circ}(q)}(n) = \frac{q^{n+1}-1}{q-1}$ and $h_{\cL^{\lambda}(q)}(n) = \frac{q^{n+1}-1}{q-1}-q$ for all $n \ge 2$. For any $q' > q$, the growth rate function of $\cL(q')$ dominates both these functions for large $n$, so the following is immediate:

\begin{theorem}
	Let $q$ be a prime power. If $\cM$ is a minor-closed class of matroids so that $h_{\cM}(n) > \frac{q^n-1}{q-1}$ for infinitely many $n$, then $h_{\cM}(n) \ge \frac{q^{n+1}-1}{q-1}-q$ for all sufficiently large $n$. 
\end{theorem}

For each integer $\ell \ge 2$, let $\cU(\ell)$ denote the class of matroids with no $U_{2,\ell+2}$-minor. Our next corollary is the main theorem of [\ref{line}].

\begin{theorem}\label{linecor}
	If $\ell \ge 2$ is an integer, then $h_{\cU(\ell)}(n) = \frac{q^n-1}{q-1}$ for all sufficiently large $n$, where $q$ is the largest prime power not exceeding $\ell$. 
\end{theorem}

Let $\Lambda_k$ denote the rank-$k$ free spike (see [\ref{govw02}] for a definition); the next corollary determines the eventual growth rate function for any class defined by excluding a free spike and a line:
\begin{theorem}\label{spikethm}
	Let $\ell \ge 2$ and $k \ge 3$ be integers. If $\cM$ is the class of matroids with no $U_{2,\ell+2}$- or $\Lambda_k$-minor, then $h_{\cM}(n) = \frac{p^n-1}{p-1}$ for all sufficiently large $n$, where $p$ is the largest prime satisfying $p \le \min(\ell,k+1)$. 
\end{theorem}

Let $\Delta_k$ denote the rank-$k$ free swirl (defined in [\ref{govw02}] as just a \emph{swirl}). We do not obtain a complete version of Theorem~\ref{spikethm} for swirls, but still obtain a result in a large range of cases. A \emph{Mersenne prime} is a prime number of the form $2^p-1$ where $p$ is also prime. 

\begin{theorem}\label{swirlthm}
		Let $2^{p}-1$ and $2^{p'}-1$ be consecutive Mersenne primes, and let $k$ and $\ell$ be integers for which  $2^p \le \ell < \min(2^{2p}+2^p,2^{p'})$ and $k \ge \max(4,2^p-2)$. If $\cM$ is the class of matroids with no $U_{2,\ell+2}$- or $\Delta_k$-minor, then $h_{\cM}(n) = \frac{2^{pn}-1}{2^p-1}$ for all sufficiently large $n$.
	\end{theorem}

If $p' > 2p$, there is a range of values of $\ell$ to which the above theorem does not apply. This does occur (for example, when $(p,p') = (127,521)$) and in fact, the growth rate function for $\cM$ can take a different eventual form for such an $\ell$; we discuss this in Section~\ref{corsection}. 

For excluding both a free spike and a free swirl, we get a nice result:

\begin{theorem}\label{spikeswirl}
	Let $\ell \ge 3$ and $k \ge 3$ be integers. If $\cM$ is the class of matroids with no $U_{2,\ell+2}$-, $\Lambda_k$- or $\Delta_k$-minor, then $h_{\cM}(n) = \tfrac{1}{2}(3^n-1)$ for all sufficiently large $n$. 
\end{theorem}


\section{The Main Result}

In this section we prove Theorem~\ref{main}. We use the notation of Oxley [\ref{oxley}], writing $\elem(M)$ for the number of points (that is, rank-$1$ flats) in a matroid $M$. The following theorem from [\ref{dhjpaper}] is our main technical tool: 

\begin{theorem}[Matroidal density Hales-Jewett theorem]\label{dhj}
	There is a function $f: \bZ^3 \times \bR \to \bZ$ so that, for every positive real number $\alpha$, every prime power $q$ and for all integers $\ell \ge 2$ and $n \ge 2$, if $M \in \cU(\ell)$ satisfies $\elem(M) \ge \alpha q^{r(M)}$ and $r(M) \ge f(\ell,n,q,\alpha)$, then $M$ has an $\AG(n-1,q)$-restriction or a $\PG(n-1,q')$-minor for some $q' > q$. 
\end{theorem}

For an integer $q \ge 2$, a matroid $M$ is \emph{$q$-dense} if $\elem(M) > \frac{q^{r(M)}-1}{q-1}$. We prove an easy lemma showing when $q$-density is lost by contraction:

\begin{lemma}\label{longline}
	Let $q \ge 2$ be an integer. If $M$ is a $q$-dense matroid and $e \in E(M)$, then either $M \con e$ is $q$-dense, or $M$ has a $U_{2,q+2}$-restriction containing $e$. 
\end{lemma}
\begin{proof}
	We may assume that $M$ is simple. If $|L| \ge q+2$ for some line $L$ through $e$, then $M$ has a $U_{2,q+2}$-restriction containing $e$. Otherwise, no line through $e$ contains $q+2$ points, so each point of $M \con e$ contains at most $q$ elements of $M$. Therefore $\elem(M \con e) \ge q^{-1} \elem(M \del e) > \frac{q^{r(M)-1}-1}{q-1}$ and $r(M \con e) = r(M) - 1$, so $M \con e$ is $q$-dense. 
\end{proof}

A simple induction now gives a corollary originally due to Kung [\ref{kung91}]:

\begin{corollary}\label{kung}
	If $\ell \ge 2$ and $M \in \cU(\ell)$ then $\elem(M) \le \frac{\ell^{r(M)}-1}{\ell-1}$.
\end{corollary}

We now reduce Theorem~\ref{main} to a case where all cocircuits are large:

\begin{lemma}\label{reduction}
	Let $t,\ell \ge 2$ be integers and let $q$ be a prime power. If $M \in \cU(\ell)$ is a $q$-dense matroid so that $(\sqrt{5}-1)^{r(M)-1} \ge \ell^{t-1}$, then $M$ has a $q$-dense restriction $M_0$ such that $r(M_0) \ge t$ and every cocircuit of $M_0$ has rank at least $r(M_0) -1$. 
\end{lemma}
\begin{proof}
	Let $\varphi = \tfrac{1}{2}(1+ \sqrt{5}).$ Let $r = r(M)$, and let $M_0$ be a minimal restriction of $M$ so that $\elem(M_0) > \varphi^{r(M_0)-r}\tfrac{q^r-1}{q-1}$. Let $r_0 = r(M_0)$. Since $(\varphi/q)^{r_0-r} \ge 1$ and $\varphi^{r_0-r} \le 1$, we have 
	\[\elem(M_0) > \tfrac{1}{q-1}\left(\varphi^{r_0-r}q^r - \varphi^{r_0-r}\right) \ge \tfrac{q^{r_0}-1}{q-1},\]
	so $M_0$ is $q$-dense. Moreover, 
	 \[\elem(M_0) > \varphi^{r_0-r}q^{r-1} \ge \varphi^{1-r}2^{r-1} = (\sqrt{5}-1)^{r-1} \ge \ell^{t-1} > \tfrac{\ell^{t-1}-1}{\ell-1},\] so $r(M_0) \ge t$ by Corollary~\ref{kung}. Finally, if $M_0$ had a cocircuit $C$ of rank at most $r(M_0) -2$, minimality would give
	\[\elem(M_0) = \elem(M_0|C) + \elem(M_0 \del C) \le (\varphi^{-2} + \varphi^{-1})\varphi^{r_0-r}\tfrac{q^r-1}{q-1}; \]
	since $\varphi^{-2} + \varphi^{-1} = 1$, this contradicts $\elem(M_0) > \varphi^{r_0-r}\tfrac{q^r-1}{q-1}$. 
\end{proof}

The next lemma finds one of two unavoidable minors in every non-$\GF(q)$-representable extension of a large projective geometry: 

\begin{lemma}\label{unavoidable}
	Let $q$ be a prime power and $m \ge 2$ be an integer. If $M$ is a non-$\GF(q)$-representable extension of $\PG(2m-1,q)$, then $M$ has a minor isomorphic to $\princext{m-1}{q}{2}$ or $\princext{m-1}{q}{m}$. 
\end{lemma}
\begin{proof}
	Since every flat in a projective geometry is modular, we know that $M$ is a principal extension of some flat $F$ of $\PG(2m-1,q)$.  Let $B$ be a basis for $\PG(2m-1)$ containing a basis $B_F$ for $F$. Since $M$ is not $\GF(q)$-representable, we have $r_M(F) \ge 2$. If $r_M(F) \ge m$, then $\si(M \con (B- I)) \cong \princext{m-1}{q}{m}$, where $I$ is an $m$-element subset of $B_F$. If $r_M(F) < m$, then $|B-B_F| \ge m-2$; it now follows that $\si(M \con (J_1 \cup J_2)) \cong \princext{m-1}{q}{2}$, where $J_1 \subseteq B_F$ and $J_2 \subseteq B-B_F$ satisfy $|J_1| = r_M(F)-2$ and $|J_2| = m-|J_1|$. 
\end{proof}

We now restate and prove Theorem~\ref{main}:

\begin{theorem}\label{maintech}
	Let $q$ be a prime power and let $m,\ell \ge 2$ be integers. If $M \in \cU(\ell)$ is $q$-dense and $r(M)$ is sufficiently large, then $M$ has a minor isomorphic to $\princext{m-1}{q}{2},\princext{m-1}{q}{m}$, or $\PG(m-1,q')$ for some $q' > q$. 
\end{theorem}
\begin{proof}
	Recall that the function $f$ was defined in Theorem~\ref{dhj}. Let $n_1 = f(\ell,2m+1,q,q^{-1})$ and let $n_0$ be an integer so that $(\sqrt{5}-1)^{n_0-1} \ge \ell^{n_1-1}$. We show that the conclusion holds whenever $r(M) \ge n_0$.
	
	Let $M \in \cU(\ell)$ be a $q$-dense matroid of rank at least $n_0$. By definition of $n_0$ and Lemma~\ref{reduction}, $M$ has a $q$-dense restriction $M_1$ such that $r(M_1) \ge n_1$ and every cocircuit of $M_1$ has rank at least $r(M_1)-1$. Note that $\elem(M_1) > q^{-1}q^{r(M_1)}$; by Theorem~\ref{dhj} and the definition of $n_1$, the matroid $M_1$ has an $\AG(2m,q)$-restriction $R$ or a $\PG(2m,q')$-minor for some $q' > q$. In the latter case, the theorem holds. In the former case, let $M_2$ be a minimal minor of $M_1$ so that
	\begin{enumerate}
		\item\label{min1} $R$ is a restriction of $M_2$, 
		\item\label{min2} every cocircuit of $M_2$ has rank at least $r(M_2)-2$, and
		\item\label{min3} $M_2$ is either $q$-dense or has a $U_{2,q+2}$-restriction. 
	\end{enumerate}
	Note that $r(M_2) \ge r(R) \ge 5$, and that contracting any element not spanned by $E(R)$ gives a matroid satisfying (\ref{min1}) and (\ref{min2}). We argue that $R$ is spanning in $M_2$; suppose not, and let $e \in E(M_2) - \cl_{M_2}(E(R))$.  If $M_2$ has a $U_{2,q+2}$-restriction $M_2|L$ containing $e$, then since $r(M_2) \ge 5$, the set $\cl_{M_2}(L)$ contains no cocircuit of $M_2$ and there is hence some $x \in E(M_2) - (\cl_{M_2}(E(R)) \cup \cl_{M_2}(L))$. Therefore $(M \con x)|L \cong U_{2,q+2}$, contradicting minimality. Thus, $M_2$ has no $U_{2,q+2}$-restriction containing $e$, so Lemma~\ref{longline} implies that $M_2 \con e$ is $q$-dense, again contradicting minimality; therefore $R$ is spanning in $M_2$.
	
	 If $x \in E(R)$, then $M_2 \con x$ is a rank-$2m$ matroid with a $\PG(2m-1,q)$-restriction; it is thus enough to show that $M_2 \con x$ is non-$\GF(q)$-representable for some such $x$, as the theorem then follows from Lemma~\ref{unavoidable}. If $M_2$ has a $U_{2,q+2}$-restriction  $M_2|L$, then any $x \in E(R) - L$ will do, since $(M_2 \con x)|L$ is not $\GF(q)$-representable. Otherwise, by Lemma~\ref{longline}, the matroid $M_2 \con x$ is $q$-dense for any $x \in E(R)$; again this implies non-$\GF(q)$-representability.
\end{proof}

\section{Lines, Spikes and Swirls}\label{corsection}

In this section, we restate and prove our four corollaries. 

\begin{theorem} If $\ell \ge 2$ is an integer, then $h_{\cU(\ell)}(n) = \frac{q^n-1}{q-1}$ for all sufficiently large $n$, where $q$ is the largest prime power not exceeding $\ell$. 
\end{theorem}
\begin{proof}
	Note that $\cL(q) \subseteq \cU(\ell)$, giving $\frac{q^n-1}{q-1} \le h_{\cU(\ell)}(n)< \infty$ for all $n$. If the result fails, then by Theorem~\ref{main2} we have either $U_{2,q^2+1} \in \cM$ or $U_{2,q'+1} \in \cM$, where $q'$ is the smallest prime power such that $q' > \ell$. Clearly $q^2 \ge q' \ge \ell+1$; it follows that $U_{2,\ell+2} \in \cU(\ell)$, a contradiction.
	\end{proof}

	 Our other corollaries depend on representability of free spikes and swirls. It can be easily shown that the free spike $\Lambda_k$ is representable over a field $\GF(q)$ if and only if there exist nonzero $\alpha_1,\alpha_2, \dotsc, \alpha_{k-1},\beta _1,\beta_2\in \GF(q)$ so that $\beta_1 \ne \beta_2$ and no sub-multiset of the $\alpha_i$ has sum equal to $\beta_1$ or $\beta_2$. The problem for $\Delta_k$ is analogous, but with products in the multiplicative group $\GF(q)^*$. Both problems are trivial unless the relevant group is of prime order, as one can choose the $\alpha_i$ in a subgroup not containing the $\beta_i$. Similarly, if the group has size at least $k+2$, one can choose the $\alpha_i$ all equal. The details for the prime-order case were dealt with in [\ref{govw02}, Lemma 11.6]; the following lemma summarises the consequences:
	
	\begin{lemma}\label{spikerep}
		If $k \ge 3$ is an integer and $q \ge 3$ is a prime power, then
		\begin{enumerate}
			\item $\Lambda_k \in \cL(q)$ and only if $q$ is composite or $k \le q-2$. 	
			\item  $\Delta_k \in \cL(q)$ if and only if $q-1$ is composite or $k \le q-3$.
		\end{enumerate}
	\end{lemma}  

	It is easy to see that $\cL^{\lambda}(q)$ contains every restriction of a matroid obtained from a matroid in $\cL(q)$ by principally truncating a line. Moreover, $\cL^{\circ}(q)$ contains all truncations of $\GF(q)$-representable matroids.  We can now show that these classes contain all free spikes: 

	\begin{lemma}\label{spikeext}
		If $q$ is a prime power and $k \ge 3$ is an integer, then $\Lambda_k \in \cL^{\lambda}(q) \cap \cL^{\circ}(q)$. 
	\end{lemma}
	\begin{proof}
		Let $G \cong K_{2,k}$ and let $M = M(G)$. The free spike $\Lambda_k$ is the truncation of the regular matroid $M$, so $\Lambda_k \in \cL^{\circ}(q)$. Let $H$ be a $K_{1,k}$-subgraph of $G$. For each prime power $q$, let $\widehat{M}$ be a $\GF(q)$-representable extension of $M$ by a point $e$ spanned by $E(H)$ but no proper subset of $E(H)$. Now we have $\Lambda_k \cong \widehat{M}' \del e$ , where $\widehat{M}'$ is obtained from $\widehat{M}$ by principally truncating the line spanned by $\{e,f\}$ for some $f \in E(H)$. Therefore $\Lambda_k \in \cL^{\lambda}(q)$.  	\end{proof}
	
	The same does not hold, however, for free swirls:
		
	\begin{lemma}\label{swirlext}
		If $q \ge 3$ is a prime power and $k \ge 4$ is an integer, then 
		\begin{itemize}
			\item $\Delta_k \in \cL^{\lambda}(q)$.
			\item $\Delta_k \in \cL^{\circ}(q)$ if and only if $\Delta_k \in \cL(q)$. 
		\end{itemize}
	\end{lemma}
	\begin{proof}
		Let $L_1,L_2, \dotsc, L_k$ be copies of $U_{2,4}$ so that $|E(L_i) \cap E(L_{i+1})| = 1$ for each $i \in \{1, \dotsc, k-1\}$ and $E(L_i) \cap E(L_j) = \varnothing$ for $|i-j| > 1$. Let $x_1 \in E(L_1) - E(L_2)$ and $x_k \in E(L_k) - E(L_{k-1})$. Let $N_k$ be defined by the repeated $2$-sum $L_1 \oplus_2 L_2 \oplus_2 \dotsc \oplus_2 L_k$. Clearly $N_k \in \cL(q)$, and $\widehat{N_k} \del \{x_1,x_k\} \cong \Delta_k$, where $\widehat{N_k}$ is the principal truncation of the line spanned by $x_1$ and $x_k$ in $N_k$ (that is, the principal extension of this line, followed by a contraction of the new element). Therefore $\Delta_k \in \cL^{\lambda}(q)$. 
		
		On the other hand, suppose that $\Delta_k$ is in exactly one of $\cL(q)$ and $\cL^{\circ}(q)$. Since $\cL(q) \subseteq \cL^{\circ}(q)$, it must be the case that $\Delta_k$ is the truncation of a rank-$(k+1)$ matroid $N \in \cL(q)$. Let $E(\Delta_k) = P_1 \cup \dotsc \cup P_k$, where the $P_i$ are pairwise disjoint two-element sets so that the union of any two cyclically consecutive $P_i$ is a circuit of $\Delta_k$, and the union of two any other $P_i$ is independent in $\Delta_k$. Since $r(N) \ge 5$ and $\Delta_k$ is the truncation of $N$, we thus have $N|(P_i \cup P_j)  = \Delta_k|(P_i \cup P_j)$ for all distinct $i$ and $j$. As $P_i \cup P_{i+1}$ is a circuit of $N$ for each $i < k$, an inductive argument gives $r_N(P_1 \cup \dotsc \cup P_{k-1}) \le k$. Similarly, $r_N(P_{k-1} \cup P_1 \cup P_k) \le 4$, so 
		$P_k \subset \cl_N(P_{k-1} \cup P_1)$ and $r(N) \le k$, a contradiction. 
	\end{proof}
		
	The fact that $\cL^{\circ}(q)$ need not contain all free swirls is the reason that Theorem~\ref{swirlthm} is more technical and less powerful than Theorem~\ref{spikethm}. We now restate and prove both these theorems:
		\begin{theorem}
		Let $k \ge 3$ and $\ell \ge 2$ be integers. If $\cM$ is the class of matroids with no $U_{2,\ell+2}$- or $\Lambda_k$-minor, then $h_{\cM}(n) = \frac{p^n-1}{p-1}$ for all sufficiently large $n$, where $p$ is the largest prime satisfying $p \le \min(\ell,k+1)$. 
		\end{theorem}
		\begin{proof}
		By Lemma~\ref{spikerep}, we have $\Lambda_k \notin \cL(p)$ and so $\cL(p) \subseteq \cM$ and $\frac{p^n-1}{p-1} \le h_{\cM}(n) < \infty$ for all $n$. If the result does not hold, then by Theorem~\ref{main2} the class $\cM$ contains $\cL^{\circ}(p), \cL^{\lambda}(p)$ or $\cL(q)$ for some prime power $q > p$. In the first two cases we have $\Lambda_k \in \cM$, a contradiction. In the last case, since $U_{2,\ell+2} \notin \cL(q)$ and $\Lambda_k \notin \cL(q)$, we know by Lemma~\ref{spikerep} that $q$ is prime and $q \le \min(\ell,k+1)$; this contradicts the maximality in our choice of $p$. 
		\end{proof}

	\begin{theorem}
		Let $2^{p}-1$ and $2^{p'}-1$ be consecutive Mersenne primes, and let $k$ and $\ell$ be integers for which  $2^p \le \ell < \min(2^{2p}+2^p,2^{p'})$ and $k \ge \max(4,2^p-2)$. If $\cM$ is the class of matroids with no $U_{2,\ell+2}$- or $\Delta_k$-minor, then $h_{\cM}(n) = \frac{2^{pn}-1}{2^p-1}$ for all sufficiently large $n$.
	\end{theorem}
	\begin{proof}
		Since $\ell \ge 2^p$ and $k \ge 2^p-2$, we have $U_{2,\ell+2} \notin \cL(2^p)$ and $\Delta_k \notin \cL(2^p)$, so $\cL(2^p) \subseteq \cM$, giving $\tfrac{2^{np}-1}{2^p-1} \le h_{\cM}(n) < \infty$ for all $n$. If the result fails, then $\cM$ contains $\cL^{\circ}(2^p), \cL^{\lambda}(2^p)$ or $\cL(q)$ for some prime power $q > 2^p$. We have $U_{2,2^{2p}+2^p+1} \in \cL^{\circ}(2^p)$, and $\Delta_k \in \cL^{\lambda}(2^p)$ by Lemma~\ref{swirlext}. If $q-1$ is composite, then $\Delta_k \in \cL(q)$. If $q-1$ is prime then it is a Mersenne prime, so $q \ge 2^{p'}$, giving $U_{2,2^{p'}+1} \in \cL(q)$. Since $\ell < \min(2^{2p}+2^p,2^{p'})$, we have $U_{2,\ell+2} \in \cM$ or $\Delta_k \in \cM$ in all cases, a contradiction. 
	\end{proof}
	
	We cannot hope for such a simple theorem applying to all $\ell$; to see why, suppose that $p' > 2p$ (for example, if $(p,p') = (127,521)$). Then if $ 2^{2p} + 2^p \le \ell < 2^{p'}$ and $k \ge 2^p-2$, it follows from Lemmas~\ref{spikerep} and~\ref{swirlext} that $\cL^{\circ}(2^p) \subseteq \cM$ but $\cL(q) \not\subseteq \cM$ for all $q > 2^p$. The Growth rate theorem thus gives $\frac{2^{p(n+1)}-1}{2^p-1} \le h_{\cM}(n) \le c \cdot 2^{pn}$ for some constant $c$, so $h_{\cM}(n)$ does not eventually equal $\frac{q^n-1}{q-1}$ for any prime power $q$. 
	
	Finally, we prove Theorem~\ref{spikeswirl}:
	
	\begin{theorem}	Let $\ell \ge 3$ and $k \ge 3$ be integers. If $\cM$ is the class of matroids with no $U_{2,\ell+2}$-, $\Lambda_k$- or $\Delta_k$-minor, then $h_{\cM}(n) = \tfrac{1}{2}(3^n-1)$ for all sufficiently large $n$. 
   \end{theorem}
   \begin{proof}
   		As before, if the theorem fails, $\cM$ contains $\cL^{\lambda}(3), \cL^{\circ}(3)$ or $\cL(q)$ for some $q > 3$. In the first two cases, we have $\Lambda_k \in \cM$, and otherwise, since either $q$ or $q-1$ is composite, we have $\Lambda_k \in \cM$ or $\Delta_k \in \cM$, a contradiction. 
   \end{proof}

\section*{Acknowledgements}
I thank Jim Geelen for bringing the corollaries regarding spikes and swirls to my attention. I also thank the two referees for their very close reading of the manuscript and their useful comments.  

\section*{References}

\newcounter{refs}

\begin{list}{[\arabic{refs}]}
{\usecounter{refs}\setlength{\leftmargin}{10mm}\setlength{\itemsep}{0mm}}

\item \label{gkw09}
J. Geelen, J.P.S. Kung, G. Whittle, 
Growth rates of minor-closed classes of matroids, 
J. Combin. Theory. Ser. B 99 (2009) 420--427.

\item\label{govw02}
J. Geelen, J.G. Oxley, D. Vertigan, G. Whittle, 
Totally free expansions of matroids, 
J. Combin. Theory. Ser. B 84 (2002), 130--179. 

\item\label{line}
J. Geelen, P. Nelson, 
The number of points in a matroid with no $n$-point line as a minor, 
J. Combin. Theory. Ser. B 100 (2010), 625--630.

\item\label{dhjpaper}
J. Geelen, P. Nelson, 
A density Hales-Jewett theorem for matroids,
J. Combin. Theory. Ser. B, to appear. 

\item\label{kung91}
J.P.S. Kung,
Extremal matroid theory, in: Graph Structure Theory (Seattle WA, 1991), 
Contemporary Mathematics, 147, American Mathematical Society, Providence RI, 1993, pp.~21--61.

\item\label{nthesis}
P. Nelson, 
Exponentially Dense Matroids. 
Ph.D. Thesis, University of Waterloo, 2011. 

\item \label{oxley}
J. G. Oxley, 
Matroid Theory (2nd edition),
Oxford University Press, New York, 2011.
\end{list}

\end{document}